\documentclass{article}
\usepackage[english]{babel}
\usepackage{amsmath,amssymb,amsthm}
\usepackage[all]{xy}
\usepackage{graphicx,graphics}

\newtheorem{theorem}{Theorem}
\newtheorem{proposition}{Proposition}
\newtheorem{lemma}{Lemma}
\theoremstyle{definition}
\newtheorem{definition}{Definition}
\newtheorem{remark}{Remark}
\newtheorem{example}{Example}

\newcommand{\gL}{\mathfrak L}
\newcommand{\gH}{\mathfrak H}
\newcommand{\id}{\mathrm{id}}
\newcommand{\Z}{\mathbb Z}
\newcommand{\Chi}{\mathcal X}

\newcommand{\tar}[2]{$\begin{array}{c} #1 \\ \mbox{#2}\end{array}$}

\begin{document}

\title{Homotopical Khovanov homology}
\author{V.O. Manturov\footnotemark[1]
, I.M. Nikonov\footnotemark[2]
} \footnotetext[1]{Bauman Moscow State Technical University, Russia;  Laboratory of
Quantum Topology, Chelyabinsk State University, Chelyabinsk, Russia}
\footnotetext[2]{Department of Mechanics and Mathematics, Moscow State University,
Russia; Faculty of Management, National Research University Higher School of Economics,
Russia}
\date{}

 \maketitle

\abstract{
 We modify the definition of the Khovanov complex for oriented links in a thickening of an oriented surface to obtain a triply graded homological link invariant with a new homotopical grading.
 }



\section{Introduction}

Virtual knot theory was discovered by Kauffman in~\cite{Kauffman}.
It turned out that virtual links are isotopy classes of links in thickened surfaces considered up to stabilizations.

Thus, virtual knots and links possess properties coming from knots as well as those coming from curves on $2$-surfaces.

It turned out~\cite{Miyazawa,DK,Man1} that this homotopical information can be converted into coefficients
of polynomials, or, in other words, polynomials can be thought of to have homotopy classes as coefficients.

Homotopy classes of curves in a closed $2$-surface can be thought of as conjugacy classes of the fundamental group of the surface. The recognition problem for such curves and groups is solved, see, e.g. \cite{HS}.

Some numerical information coming from these homotopy classes can be converted into additional gradings of the Khovanov homology~\cite{DKM,Man2}.

In the present paper, we show how the homotopical information about curves in Kauffman states
can be converted into additional gradings.

Thus, curves which appear as summands (with some other coefficients) are raised to the other level.
Though some factorizations in the grading space are needed for the complex to be well defined, the homotopy information which is converted into gradings, does not reduce to some homology classes or numerical grading; highly non-commutative group theory comes into play in the new form of gradings of the Khovanov complex.

\section{Homotopical Khovanov homology}

Let $S$ be a connected closed oriented two dimensional surface. We consider links in the thickening $S\times [0,1]$ of the surface. Such links can be described by diagrams look like $4$-valent graphs embedded into $S$ with the structure of over- and undercrossing in the vertices of the graphs. Different diagrams of the same link differ by diagram isotopies and Reidemeister moves.

\begin{definition}
Let $D$ be a link diagram in $S$. A {\em source-sink structure} on $D$ is an orientation of edges of $D$ (considered as a $4$-valent graph) such that at each crossing of $D$ there are two opposite incoming edges and two opposite outcoming edges (see Fig.~\ref{fig:source-sink_structure}).

\begin{figure}[h]
\centering\includegraphics[width=0.15\textwidth]{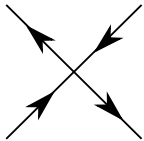}
\caption{Source-sink structure} \label{fig:source-sink_structure}
\end{figure}
\end{definition}

Many invariants of classical knots and links can be extended to the case of links in the surface $S$ with minor changes of construction or without changes at all. Let us recall the construction of Khovanov homology of links. For simplicity we work over $\Z_2$.

Let $D$ be a diagram of some oriented link $L$ in the surface $S$. Let $\Chi(D)$ be the set of crossing of the diagram $D$. At each crossing of $D$ there are two possible resolutions $0$ and $1$ (see Fig.~\ref{fig:resolution}). A {\em state} is a map $s\colon \Chi(D)\to \{0,1\}$ (or equivalently a sequence of $0$s and $1$s indexed by crossings of $D$). Any state $s$ determines resolutions at all crossings of the diagram $D$. The diagram $D_s$ which appears after resolution at each crossing according to the state $s$, has no crossings, and therefore it is a set of nonintersecting circles (closed simple curves) in $S$. Let $\gamma(s)$ be the number of circles in $D_s$ and $\beta(s)$ be the number of $1$s in $s$.

\begin{figure}[h]
\centering\includegraphics[width=0.6\textwidth]{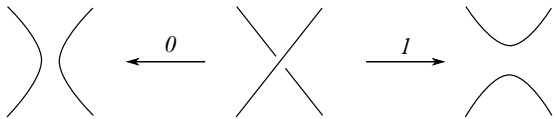}
\caption{Resolutions of a crossing} \label{fig:resolution}
\end{figure}

Thus, $D$ has $2^n$ states, where $n$ is the number of crossings in $D$. The set of states $\{0,1\}^{\Chi(D)}$ can be interpreted as the set of vertices of a $n$-dimensional cube ({\em the state cube}). The edges of the cube are oriented from the state $(0,0,\dots,0)$ to the state $(1,1,\dots,1)$.

Let $V$ be the two dimensional graded vector space with basis $v_-$ and $v_+$. Set the grading of the basis vectors to be $\deg(v_\pm)=\pm 1$. Let the maps $m\colon V\otimes V\to V$ and $\Delta\colon V\to V\otimes V$ be given by formulas

\begin{equation}\label{eq:mDelta_formula}
\begin{array}{lll}
m(v_+\otimes v_+)=v_+, & m(v_+\otimes v_-)=v_-, & \Delta(v_-)=v_-\otimes v_-\\
m(v_-\otimes v_+)=v_-, & m(v_-\otimes v_-)=0, & \Delta(v_+)=v_+\otimes v_- - v_-\otimes v_.
\end{array}
\end{equation}

With each state $s$, assign the state  a space $V(s)=V^{\otimes \gamma(s)}$.
Now, the chain space of the Khovanov complex, we are going to construct for $D$, will be the sum
$$[[D]]=\bigoplus_{s\in \{0,1\}^{\Chi(D)}}V(s).$$

The space $[[D]]$ has two gradings: {\em the homological grading}
$\beta$ and {\em the quantum grading} $q$. If $x=x_1\otimes
x_2\otimes\dots\otimes x_{\gamma(s)}\in V(s)\subset [[D]]$ where
$x_i=v_\pm, i=1,\dots,\gamma(s),$ then we set $\beta(x)=\beta(s)$ and
$q(x)=\sum_{i=1}^{\gamma(s)}\deg(x_i)+\beta(s)$.

Any edge $e = s\to s'$ connects two states $s$ and $s'$. The
resolutions $D_s$ and $D_{s'}$ differ at one
crossing. If $D$ admits a source-sink structure then one of the following pictures can happen: either a circle can split into two circles or two circles
can merge into one circle (see Fig.~\ref{fig:circle_resolution}).

\begin{figure}[h]
\centering\includegraphics[width=0.6\textwidth]{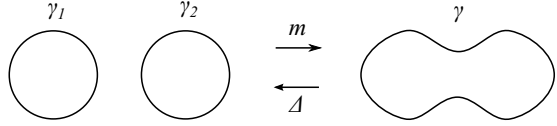}
\caption{Transformations of resolution}
\label{fig:circle_resolution}
\end{figure}

In the first case, one should consider the map
$\Delta\otimes\id^{\otimes\gamma(s)-1}\colon V(s)\to V(s')$ where
$\Delta$ acts on the factor $V$ which corresponds to the
splitting circle. In the second case, one should consider the map
$m\otimes\id^{\otimes\gamma(s)-2}\colon V(s)\to V(s')$ where $m$
acts on the factors $V\otimes V$ which correspond to the merging
circles. The considered map from $V(s)$ to $V(s')$ is called the {\em
partial differential} corresponding to the edge $e$ of the Khovanov complex and denoted as
$\partial_e$.

The sum of partial differentials for all the edges of the cube of
states $d=\sum_{e=s\to s'}\partial_e$ is the differential of
Khovanov complex. The differential does not change quantum grading and increases
homology grading by $1$.

In order to construct an invariant homology theory, we should ensure among other thing invariance under first Reidemeister moves.
To this end, let us shift the homological grading by
 $-n_-$ and the quantum grading by $n_+-2n_-$ where $n_+$ and
$n_-$ are the numbers of positive and negative crossings in the
diagram $D$. Let $C(D)$ be the shifted complex. Then homology
$Kh(D)=H(C(D),d)$ is the {\em Khovanov homology} of the link $L$.

\begin{definition}
Let $S$ be a closed two-dimensional surface.  We consider the set $\mathfrak L = [S^1;S]$ of all the homotopy classes of free oriented loops in $S$. Let $\bigcirc\in\mathfrak L$ be the homotopy class of contractible loops.

For any closed curve $\gamma$ one can consider the curve $-\gamma$ obtained from $\gamma$ by the orientation change.
Let $\gH$ be the quotient group of the free abelian group with generator set $\gL$ modulo the relations $\bigcirc = 0$ and $[\gamma]=[-\gamma]$ for all free loops $\gamma$.
\end{definition}


We consider a new homotopical grading in the complex $[[L]]$
valued in the group $\gH$. Let $s$ be a state and
$C_1,C_2,\dots,C_{\gamma(s)}$ be the circles of the resolution
$L_s$. For each element $x=x_1\otimes x_2\otimes\dots\otimes
x_{\gamma(s)}\in V(s)$ where $x_i=v_\pm$ corresponds to the circle
$C_i, i=1,\dots,\gamma(s)$, we define its {\em homotopical grading}
as
$$h(x)=\sum_{i=1}^{\gamma(s)}\deg(x_i) [C_i]\in\gH.$$

Let $d_h$ be the "part" of the differential $d$ that does not change the homotopical grading (i.e. the composition of $d$ with the projection to the subspace with the correspondent homotopical grading).
The map $d_h$ is presented as a sum $d_h=\sum_{e\in E} (\partial_e)_h$ over all the edges of the state cube. For any edge $e = s\to s'$ the partial map $(\partial_e)_h$ has the form $m_h\otimes \id^{\otimes \gamma(s)-2}$ if the partial differential is equal to  $m\otimes \id^{\otimes \gamma(s)-2}$ and
$(\partial_e)_h$ has the form $\Delta_h\otimes \id^{\otimes \gamma(s)-1}$ if the partial differential is equal to  $\Delta\otimes \id^{\otimes \gamma(s)-1}$.
Here $m_h$ and $\Delta_h$ are the parts of the maps $m$ and $\Delta$ that preserve the homotopical grading.

The formula for the map $m_h\colon V\otimes V\to V$, which corresponds to  the bifurcaction of two circles $\gamma_1$ and $\gamma_2$ to a circle $\gamma$, depends on the homotopy classes of the circles:
\begin{equation}\label{eq:mh_formula}
m_h = \left\{\begin{array}{cl}
m, & [\gamma_1]=[\gamma_2]=[\gamma]=\bigcirc;\\
m^1_h, & [\gamma_2]=\bigcirc, [\gamma_1]=[\gamma]\ne \bigcirc;\\
m^2_h, & [\gamma_1]=\bigcirc, [\gamma_2]=[\gamma]\ne\bigcirc;\\
m^0_h, & [\gamma]= \bigcirc, [\gamma_1],[\gamma_2]\ne\bigcirc;\\
0, & [\gamma_1],[\gamma_2],[\gamma]\ne\bigcirc.
\end{array}\right.
\end{equation}
The maps $m^0_h, m^1_h, m^0_h$ are defined by the formulas
\begin{equation}\label{eq:mh_formula_monoms}
\begin{array}{lll}
m^0_h(v_+\otimes v_+)=0,   & m^1_h(v_+\otimes v_+)=v_+, & m^2_h(v_+\otimes v_+)=v_+,\\
m^0_h(v_+\otimes v_-)=v_-, & m^1_h(v_+\otimes v_-)=0, & m^2_h(v_+\otimes v_-)=v_-,\\
m^0_h(v_-\otimes v_+)=v_-, & m^1_h(v_-\otimes v_+)=v_-, & m^2_h(v_-\otimes v_+)=0,\\
m^0_h(v_-\otimes v_-)=0,   & m^1_h(v_-\otimes v_-)=0, & m^2_h(v_-\otimes v_-)=0.
\end{array}
\end{equation}

Analogously, the formula for the map $\Delta_h\colon V\to V\otimes V$, which corresponds to  transformations of a circle $\gamma$ to two circles $\gamma_1$ and $\gamma_2$, depends on the homotopical classes of the circles as follows:
\begin{equation}\label{eq:Deltah_formula}
\Delta_v = \left\{\begin{array}{cl}
\Delta, & [\gamma_1]=[\gamma_2]=[\gamma]=\bigcirc;\\
\Delta^1_h, & [\gamma_2]=\bigcirc, [\gamma_1]=[\gamma]\ne \bigcirc;\\
\Delta^2_h, & [\gamma_1]=\bigcirc, [\gamma_2]=[\gamma]\ne\bigcirc;\\
\Delta^0_h, & [\gamma]= \bigcirc, [\gamma_1],[\gamma_2]\ne\bigcirc;\\
0, & [\gamma_1],[\gamma_2],[\gamma]\ne\bigcirc.
\end{array}\right.
\end{equation}

The maps $\Delta^0_h, \Delta^1_h, \Delta^0_h$ are defined by the formulas
\begin{equation}\label{eq:Deltah_formula_monoms}
\begin{array}{lll}
\Delta^0_h(v_+)=v_+\otimes v_- + v_-\otimes v_+,   & \Delta^1_h(v_+)=v_+\otimes v_-, & \Delta^2_h(v_+)=v_-\otimes v_+,\\
\Delta^0_h(v_-)=0,   & \Delta^1_h(v_-)=v_-\otimes v_-, & \Delta^2_h(v_-)=v_-\otimes v_-.
\end{array}
\end{equation}

\begin{remark}
The formulas~(\ref{eq:mh_formula}),(\ref{eq:Deltah_formula}) show that a partial map is not zero only if one of the transformed circles is not trivial.
\end{remark}

\begin{theorem}\label{thm:homotopical_differential}
The map $d_h$ is a differential.
\end{theorem}

\begin{proof}
In order to show that the map $d_h$ is a differential we need to check that every $2$-face in the state cube of the diagram $D$ is commutative. Any $2$-face corresponds to a resolution switch of two crossings of the diagram whereas the resolution of the other crossings is fixed. Since any resolution of all the crossings of the diagram except the chosen two crossings yields a diagram $D'$ with two crossings and source--sink structure on it, there three possible configurations of the two remaining crossings' position (see Fig.~\ref{fig:three_configurations}). Note that any (partial) resolution of a diagram with source-sink structure inherits source-sink structure from the diagram. Here we consider only the components of the diagram $D'$ which contain the crossings and don't point out  under- and overcrossings structure because it doesn't change the states of the $2$-face (but determines the maps of the face).

\begin{figure}[h]
\centering
\begin{tabular}{ccc}
\includegraphics[width=0.25\textwidth]{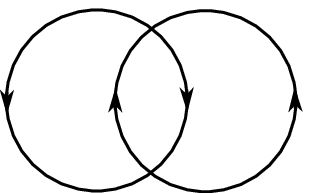} &
\includegraphics[width=0.25\textwidth]{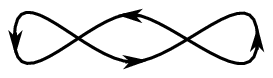} &
\includegraphics[width=0.25\textwidth]{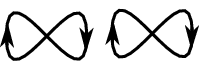} \\
Case I & Case II & Case III
\end{tabular}
\caption{Diagrams with two crossings} \label{fig:three_configurations}
\end{figure}

If the crossings belong to different components (case III) then the $2$-face is obviously commutative.

Case I. Among the resolutions of the face there are two states with two circles and two states with one circle (see Fig.~\ref{fig:caseI_resolutions}). Thus, we have two subcases depending on how many circles are there in the initial state: with one circle (for example, $e$) or with two circles ($a$ and $b$). The corresponding diagrams are shown in Fig.~\ref{fig:caseI_diagrams}. Now we should check all the variants, which homotopy classes of the circles $a,b,c,d,e,f$ are trivial and which are not.

\begin{figure}[h]
\centering\includegraphics[width=0.4\textwidth]{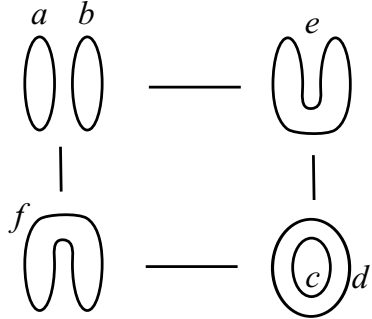}
\caption{Case I: Resolutions} \label{fig:caseI_resolutions}
\end{figure}

\begin{figure}[h]
\centering
 \tar{
  \xymatrix{ V\ar[r]^{\Delta_h} \ar[d]_{\Delta_h} & V^{\otimes 2} \ar[d]^{m_h} \\
             V^{\otimes 2}\ar[r]^{m_h} & V}
 }{Case A}\quad
 \tar{
  \xymatrix{ V^{\otimes 2}\ar[r]^{m_h} \ar[d]_{m_h} & V \ar[d]^{\Delta_h} \\
             V\ar[r]^{\Delta_h} & V^{\otimes 2}}
 }{Case B}
\caption{Case I: Face diagrams} \label{fig:caseI_diagrams}
\end{figure}
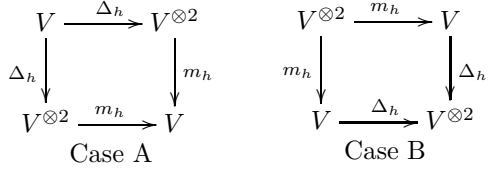

Case I.A. The initial state of the face has one circle.

Let $k$ be the number of homotopically trivial circles in the pair $a,b$, and $l$ be the number of homotopically trivial circles in the pair $c,d$. Then
$0\le k,l\le 2$. Without loss of generality assume that $k\le l$.

If $k=l$ then the $2$-face is commutative due to symmetry.

Let $k=0, l=1$. Assume that the circle $c$ is trivial and $d$ is not trivial. Then $[e]=[f]=[d]$ so $e$ and $f$ are not trivial. Hence, the commutativity relation reduces to the equality $m^2_h\Delta^2_h=0$. But $m^2_h\Delta^2_h(v_\pm)=m^2_h(v_-\otimes v_\pm)=0$, so the face is commutative.

Let $k=0, l=2$. Then $e$ and $f$ are trivial as connected sums of trivial circles $c$ and $d$. Hence, we need check the equality $m^0_h\Delta^0_h=m\Delta$. But the both maps are zero (over $\Z_2$), so the face is commutative.

The case $k=1, l=2$ is impossible. Indeed, let $[a]=\bigcirc$ and $[b]\ne\bigcirc$. Then $[e]=[b]$ is not trivial. On the other hand $e$ must be trivial as a
as a connected sum of trivial circles $c$ and $d$.

Case I.B. The initial state of the face has two circles.

If $e$ and $f$ are either both trivial or both nontrivial then the face is commutative due to symmetry.

Let $[e]=\bigcirc$ and $[f]\ne\bigcirc$. Then the circles $a,b,c,d$ must be all nontrivial. Commutativity of the face follows from the equality
$\Delta^0_h m^0_h=0$.

Case II. Among the resolutions of the face there are one state with one circle, two states with two circles and two states with one circle and one state with three circles (see Fig.~\ref{fig:caseII_resolutions}). There are three subcases (A,B,C) depending on how many circles are there in the initial state. The corresponding diagrams are shown in Fig.~\ref{fig:caseII_diagrams}.

\begin{figure}[h]
\centering\includegraphics[width=0.4\textwidth]{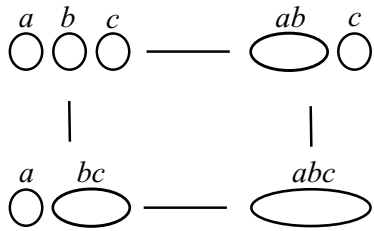}
\caption{Case II: resolutions} \label{fig:caseII_resolutions}
\end{figure}

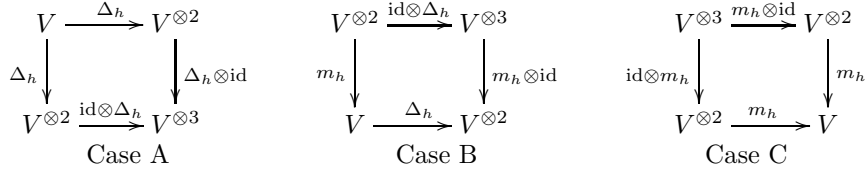
\begin{figure}[h]
\centering
 \tar{
  \xymatrix{ V\ar[r]^{\Delta_h} \ar[d]_{\Delta_h} & V^{\otimes 2} \ar[d]^{\Delta_h\otimes\id} \\
             V^{\otimes 2}\ar[r]^{ \id\otimes \Delta_h} & V^{\otimes 3}}
 }{Case A}\quad
 \tar{
  \xymatrix{ V^{\otimes 2}\ar[r]^{\id\otimes\Delta_h} \ar[d]_{m_h} & V^{\otimes 3} \ar[d]^{m_h\otimes\id} \\
             V\ar[r]^{\Delta_h} & V^{\otimes 2}}
 }{Case B}\quad
 \tar{
  \xymatrix{ V^{\otimes 3}\ar[r]^{m_h\otimes\id} \ar[d]_{\id\otimes m_h} & V^{\otimes 2} \ar[d]^{m_h} \\
             V^{\otimes 2}\ar[r]^{m_h} & V}
 }{Case C}
\caption{Case II: face diagrams} \label{fig:caseII_diagrams}
\end{figure}

Let us enumerate possible variants when some of the circles $a,b,c,ab,bc,abc$ are homotopically trivial:
\begin{enumerate}
\item  The circles $a,b,c$ are trivial.

  Then the circles $ab,bc,abc$ are trivial and we have a face of the usual Khovanov complex.

\item Among $a,b,c$ there is only one nontrivial circle.
 \begin{enumerate}
    \item $[a]\ne\bigcirc$
    \item $[b]\ne\bigcirc$
    \item $[c]\ne\bigcirc$
 \end{enumerate}

  In this case $[abc]\ne\bigcirc$, the circle $ab$ is not trivial if $[a]\ne\bigcirc$ or $[b]\ne\bigcirc$, and the circle $bc$ is not trivial if $[b]\ne\bigcirc$ or $[c]\ne\bigcirc$.

\item Among $a,b,c$ only one circle is trivial.
  \begin{enumerate}
    \item $[a]=\bigcirc$. Then $[ab]=[b]$ is not trivial and $[abc]=[bc]$.
     \begin{enumerate}
        \item $[bc]=\bigcirc$
        \item $[bc]\ne\bigcirc$
     \end{enumerate}
    \item $[b]=\bigcirc$. Then $[ab]=[a]$, $[bc]=[c]$ are not trivial.
     \begin{enumerate}
        \item $[abc]=\bigcirc$
        \item $[abc]\ne\bigcirc$
     \end{enumerate}
    \item $[c]=\bigcirc$. Then $[bc]=[b]$ is not trivial and $[abc]=[ab]$.
     \begin{enumerate}
        \item $[ab]=\bigcirc$
        \item $[ab]\ne\bigcirc$
     \end{enumerate}
 \end{enumerate}

\item None of $a,b,c$ is trivial.
  \begin{enumerate}
    \item $[ab]=[bc]=\bigcirc$. Then $[abc]=[a]=[c]$ is not trivial.
    \item $[ab]=\bigcirc, [bc]\ne\bigcirc$. Then $[abc]=[c]$ is not trivial.
    \item $[ab]\ne\bigcirc, [bc]=\bigcirc$. Then $[abc]=[a]$ is not trivial.
    \item $[ab],[bc]\ne\bigcirc$.
     \begin{enumerate}
        \item $[abc]=\bigcirc$
        \item $[abc]\ne\bigcirc$
     \end{enumerate}
 \end{enumerate}

\end{enumerate}

Any combination of cases A,B,C and 1--4.d.ii determines the maps of the face. Thus, we
can write down commutativity relation of the face in table form (see Table~1). It is easy
to check that all the equalities of the table hold, so the case II is also commutative.
The proof is finished.

\begin{table}\label{tab:caseII_relation_table}
\center{Table 1. Commutativity relation.}{ \tiny
\begin{tabular}{|l|c|c|c|} 
\hline
 $ $ & A & B &C \\
\hline
1 & \multicolumn{3}{c|}{as in usual Khovanov complex} \\
\hline
2.a & $(\Delta^1_h\otimes\id)\Delta^1_h=(\id\otimes\Delta)\Delta^1_h$
    & $(m^1_h\otimes\id)(\id\otimes\Delta)=\Delta^1_h m^1_h$
    & $m^1_h(m^1_h\otimes\id)= m^1_h(\id\otimes m)$
    \\ \hline
2.b & $(\Delta^2_h\otimes\id)\Delta^1_h=(\id\otimes\Delta^1_h)\Delta^2_h$
    & $(m^2_h\otimes\id)(\id\otimes\Delta^1_h)=\Delta^1_h m^2_h$
    & $m^1_h(m^2_h\otimes\id)= m^2_h(\id\otimes m^1_h)$
    \\ \hline
2.c & $(\Delta\otimes\id)\Delta^2_h=(\id\otimes\Delta^2_h)\Delta^2_h$
    & $(m\otimes\id)(\id\otimes\Delta^2_h)=\Delta^2_h m^2_h$
    & $m^2_h(m\otimes\id)= m^2_h(\id\otimes m^2_h)$
    \\ \hline
3.a.i  & $(\Delta^2_h\otimes\id)\Delta^0_h=(\id\otimes\Delta^0_h)\Delta$
       & $(m^2_h\otimes\id)(\id\otimes\Delta^0_h)=\Delta^0_h m$
       & $m^0_h(m^2_h\otimes\id)= m(\id\otimes m^0_h)$
       \\ \hline
3.a.ii & $0=0$
       & $0=0$
       & $0=0$
       \\ \hline

3.b.i  & $(\Delta^1_h\otimes\id)\Delta^0_h=(\id\otimes\Delta^2_h)\Delta^0_h$
       & $(m^1_h\otimes\id)(\id\otimes\Delta^2_h)=\Delta^0_h m^0_h$
       & $m^0_h(m^1_h\otimes\id)= m^0_h(\id\otimes m^2_h)$
       \\ \hline
3.b.ii & $0=0$
       & $(m^1_h\otimes\id)(\id\otimes\Delta^2_h)=0$
       & $0= 0$
       \\ \hline
3.c.i  & $(\Delta^0_h\otimes\id)\Delta=(\id\otimes\Delta^1_h)\Delta^0_h$
       & $(m^0_h\otimes\id)(\id\otimes\Delta^1_h)=\Delta m^0_h$
       & $m(m^0_h\otimes\id)= m^0_h(\id\otimes m^1_h)$
       \\ \hline
3.c.ii & $0=0$
       & $0=0$
       & $0=0$
       \\ \hline

4.a    & $(\Delta^0_h\otimes\id)\Delta^2_h=(\id\otimes\Delta^0_h)\Delta^1_h$
       & $(m^0_h\otimes\id)(\id\otimes\Delta^0_h)=\Delta^2_h m^1_h$
       & $m^2_h(m^0_h\otimes\id)= m^1_h(\id\otimes m^0_h)$
       \\ \hline
4.b    & $(\Delta^0_h\otimes\id)\Delta^2_h=0$
       & $0=0$
       & $m^2_h(m^0_h\otimes\id)= 0$
       \\ \hline
4.c    & $0=(\id\otimes\Delta^0_h)\Delta^1_h$
       & $0=0$
       & $0= m^1_h(\id\otimes m^0_h)$
       \\ \hline
4.d.i  & $0=0$
       & $0=\Delta^0_h m^0_h$
       & $0= 0$
       \\ \hline
4.d.ii & \multicolumn{3}{c|}{$0=0$}
       \\ \hline

\end{tabular}}
\end{table}
\end{proof}

\begin{remark}
1. The differentials $d$ and $d_h$ don't commute. Indeed, if this were a case then the difference $d_v=d-d_h$ would have been a differential. But in the case I.B considered above in Theorem~\ref{thm:homotopical_differential} with trivial $e$ and nontrivial $f$ we have $v_+\otimes v_-\mapsto 0\mapsto 0$ for the path of the face that goes through $e$, and $v_+\otimes v_-\mapsto v_-\mapsto v_-\otimes v_-$ for the path of the face that goes through $f$. So, the face is not commutative for the map $d_v$.

2. If we set $m^0_h=m^1_h=m^2_h=0$  and
$\Delta^0_h=\Delta^1_h=\Delta^2_h=0$ in
formulas~\eqref{eq:mh_formula}, \eqref{eq:Deltah_formula} the
modified map $d_h$ will also be a differential but the complex will
not be invariant under Reidemeister moves.
\end{remark}

Thus, $([[D]], d_h)$ is a chain complex with triple gradings. We use the conventional shift by $-n_-$ of homological grading and the shift by $n_+-2n_-$ of quantum grading (where $n_+$ and $n_-$ are the numbers of positive and negative crossings in the diagram) to obtain a chain complex $C(D)$
whose homology $Kh_h(D)=H(C(D),d_h)$ we call {\em homotopical Khovanov homology}. It is an abelian group with three gradings
$$
Kh_h(D)=\sum_{i\in\Z, j\in\Z,\mathfrak h\in\gH} Kh_h(D)_{i,j,\mathfrak h}
$$
(here $i$ is the homology grading, $j$ is the quantim grading and $\mathfrak h$ is the homotopical grading)
which turns out to be a link invariant like ordinary Khovanov homology.

\begin{theorem}\label{thm:homotopical_invariance}
Homology $Kh_h$ is an invariant of oriented links which admits source-sink structure.
\end{theorem}

\begin{proof}

Let $D$ be a diagram of an oriented link with source-sink structure and let $D'$ be the diagram that obtained from $D$ by a first Reidemeister move. Then the complex $[[D']]$ looks like $[[D]]\stackrel{\Delta_h}{\longrightarrow} [[D\sqcup\bigcirc]]$ (if the new crossing is negative) or  $[[D\sqcup\bigcirc]]\stackrel{m_h}{\longrightarrow} [[D]]$ (if the new crossing is positive). The complex $[[D\sqcup\bigcirc]]$ is isomorphic to $[[D]]\otimes V$. It contains a cubcomplex $[[D\sqcup\bigcirc]]_{v_+}=[[D]]\otimes v_+$ and we denote $[[D\sqcup\bigcirc]]_{v_+=0}=[[D\sqcup\bigcirc]]/[[D]]_{v_+}$ the quotient complex by this subcomplex.

\begin{lemma}\label{lem:homotopical_invariance}
The maps $[[D\sqcup\bigcirc]]_{v_+}\stackrel{m_h}{\longrightarrow} [[D]]$ and $[[D]] \stackrel{\Delta_h}{\longrightarrow} [[D\sqcup\bigcirc]]_{v_+=0}$ are isomorphisms of complexes.
\end{lemma}
\begin{proof}
Since the small circle is trivial, $m_h=m$ or $m_h=m^1_h$. In both cases $m_h$ establishes an isomorphism between $V\otimes v_+$ and $V$ that extends to the desired isomorphism of complexes. For the second map we have $\Delta_h=\Delta$ or $\Delta^1_h$ which define an isomorphism between $V$ and $V\otimes V/(v_+=0)$.
\end{proof}

With the above lemma we can use the proof of invariance of the usual Khovanov homology (see, for example,~\cite{BN}) without any changes.
\end{proof}

\begin{remark}
In fact, homotopical Khovanov homology can be defined for links without source-sink
structure. Indeed, in general case an additional (nonorientable) bifurcation type appears
where one circle transforms to another circle (see Fig.~\ref{fig:nonorient_bifurcation}).

\begin{figure}[h]
\centering\includegraphics[width=0.6\textwidth]{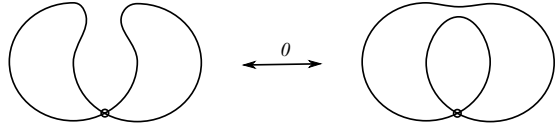}
\caption{Nonorientable bifurcation of a cirle}
\label{fig:nonorient_bifurcation}
\end{figure}

The circled crossings in Fig.~\ref{fig:nonorient_bifurcation},\ref{fig:nonorient_2_faces}
are considered as virtual but not real crossings, i.e. as additional intersections which
appear by projection of the resolution into the plane.

In the state cube this bifurcation corresponds to  zero map. So, in order to show that
homotopical Khovanov complex is well defined we must check $2$-faces which include the
new type of state transformation (see Fig.~\ref{fig:nonorient_2_faces}).

\begin{figure}[h]
\centering
\begin{tabular}{ccc}
\includegraphics[width=0.3\textwidth]{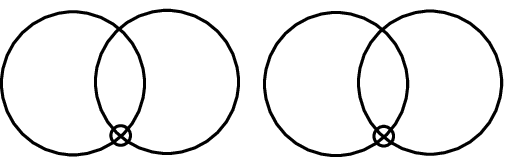} &
\includegraphics[width=0.3\textwidth]{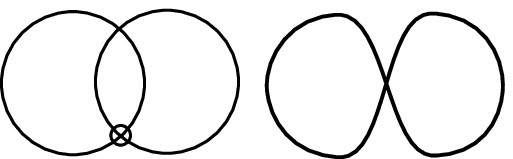} &
\includegraphics[width=0.2\textwidth]{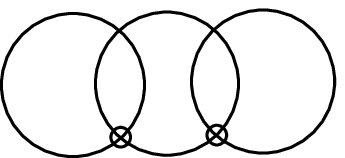} \\
Case IV & Case V & Case VI
\end{tabular}
\begin{tabular}{cc}
\includegraphics[width=0.25\textwidth]{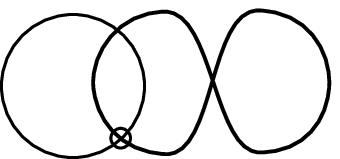} &
\includegraphics[width=0.17\textwidth]{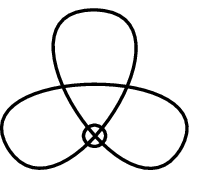} \\
Case VII & Case VIII
\end{tabular}
\caption{Nonorientable diagrams with two crossings} \label{fig:nonorient_2_faces}
\end{figure}

The commutativity relation for these faces is reduced to the equality $0=0$ for all cases except one (see Fig.~\ref{fig:caseVIII}) where it takes the form $m_h\Delta_h=0$. But this equality holds because $m\Delta=m^0_h\Delta^0_h=m^1_h\Delta^1_h=m^2_h\Delta^2_h=0$.

\begin{figure}[h]
\centering
$  \xymatrix{ V\ar[r]^{\Delta_h} \ar[d]_{0} & V^{\otimes 2} \ar[d]^{m_h} \\
             V\ar[r]^{0} & V}
$
\caption{Case VIII: nontrivial face diagram} \label{fig:caseVIII}
\end{figure}

The proof of invariance of homotopical Khovanov homology needs no corrections since it does not encounter the new bifurcation.

Thus, the following theorem holds.
\begin{theorem}
Homology $Kh_h$ is an invariant of oriented links.
\end{theorem}
\end{remark}

Let us mention some simple properties of homotopical Khovanov homology.

\begin{proposition}\label{prop:hkh_classical_link}
If $L$ is a classical link in $S$, i.e. there is a $3$-ball $B\subset S\times[0,1]$ which contains $L$, then homotopical Khovanov homology of $L$ coincides with the usual Khovanov homology of $L$.
\end{proposition}
\begin{proof}
If $L$ is a classical link then the homotopical grading of all elements in the Khovanov complex is zero, so maps $m_h$ and $\Delta_h$ coincide with $m$ and $\Delta$. Hence, in this case we have the usual Khovanov complex.
\end{proof}

\begin{proposition}\label{prop:hkh_symmetry}
Let $L$ be an oriented link in $S$. Let $-L$ differ from $L$ by orientation change
and $\check{L}$ be the link with the opposite source-sink structure. Then $Kh_h(-L)_{i,j,\mathfrak h}=Kh_h(L)_{i,j,\mathfrak h}$
, $Kh_h(\bar L)_{i,j,\mathfrak h}=Kh_h(L)_{-i,-j,-\mathfrak h}$
and $Kh_h(\check L)_{i,j,\mathfrak h}=Kh_h(L)_{i,j,\mathfrak h}$ for all $i\in\Z, j\in\Z,\mathfrak h\in\gH$. In other words, homotopical Khovanov homology does not depend on the choice of source-sink structure and does not detect invertibility of knots.
\end{proposition}
The proof of the statement does not differ from the the proof for the usual Khovanov homology.


\begin{example}[Simple closed curves in the surface]
Let $\gamma$ be a simple closed curve in $S$. It can be considered as a knot diagram without crossings. Usual Khovanov homology of $\gamma$ is $Kh(\gamma)=\Z_2\oplus\Z_2$ with homological grading equal to $0$ and quantum gradings equal to $\pm 1$. This homology does not depend on how the curve lie in the surface. Homotopical Khovanov homology is more sensitive. $Kh_h(\gamma)=\Z_2\oplus\Z_2$ where the first summand has homological grading $0$, quantum grading $-1$ and homotopical grading $-[\gamma]$, and the second summand has homological grading $0$, quantum grading $1$ and homotopical grading $[\gamma]$. Thus, the homotopy type of $\gamma$ is kept in the homotopical grading.

\begin{theorem}\label{thm:hkh_simple_curve}
Let knots $K$ and $K'$ in $S$ have complexity $0$, i.e. admit diagrams without crossings. Then $K$  isotopic to $K'$ (up to orientation reversion) if and only if $Kh_h(K)=Kh_h(K')$.
\end{theorem}
\begin{proof}
Let $K$ and $K'$ are presented by simple closed curve in $S$. If
$Kh_h(K)=Kh_h(K')$ then  $[K]=[K']\in\gH$ and $[K]=[\pm K']\in\gL$.
Hence there is a homotopy between simple closed curves $K$ and $K'$.
Then by Baer--Zieschang theorem there is an isotopy of $S$ that
sends $K$ to $K'$.
\end{proof}
\end{example}

\begin{example}[Knots in torus]
Let $K$ be a knot in the thickening of torus $T^2$. For torus we have $\gL=H_1(T^2,\Z)$. Let $D$ be a diagram of $K$ in the torus.

For any state $s$ the resolution $D_s$ is a set of nonintersecting circles in $T^2$. If $\gamma_1$ and $\gamma_2$ are nontrivial circles in $D_s$ then their homology classes coincide up to sign. Then homotopy gradings of elements in $V(s)$ equal to $k[\gamma_1], k\in\Z$. If the homology class $[D]=[K]$ of the source-sink structure of the knot is not zero the gradings will be proportional to $[K]$. If the homology class $[D]$ is trivial then gradings can be nonproportional but must be even.

Consider a link $L$ in the torus with two crossings (see Fig.~\ref{fig:torus_knot}). Homotopical Khovanov homology of the link $L$ are isomorphic to $(\Z_2)^8$. The gradings of generators of $Kh_h(L)$ are $(-1,-2,0)$, $(0,0,0)$, $(0,0,0)$, $(1,2,0)$ and $(0,2,2[\lambda])$, $(0,-2,-2[\lambda])$, $(0,2,2[\mu])$, $(0,-2,-2[\mu])$ where $\lambda$ and $\mu$ are the longitude and the meridian of the torus. Remark that the link $L$ as a virtual link is a virtualization of the trivial link with two components. Hence, the usual Khovanov homology of $L$ is $Kh(L)=(\Z_2)^4$ with gradings of generators equal to $(0,-2),(0,0),(0,0),(0,2)$.

\begin{figure}[h]
\centering\includegraphics[width=0.2\textwidth]{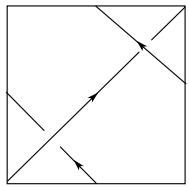}
\caption{Link in the torus} \label{fig:torus_knot}
\end{figure}
\end{example}

\section*{Acknowledgments}
The authors were partially supported with grants RFBR-13-01-00830, RFBR-14-01-91161 and
RFBR-14-01-31288 and with grant of RF President NSh -- 581.2014.1. The first author is
partially supported by Laboratory of Quantum Topology of Chelyabinsk State University
(Russian Federation government grant 14.Z50.31.0020)


\begin{thebibliography}{99}

 \bibitem{BN}
D.~Bar--Natan, On Khovanov's categorification of the Jones
polynomial, {\em Algebraic and Geometric Topology}, {\bf 2}:16
(2002), pp.\ 337--370.

\bibitem{DK} H.A. Dye, L.H. Kauffman,  Virtual crossing number and the arrow polynomial, {\em J. Knot Theory Ramifications}, {\bf 18}:10 (2009), pp.\ 1335-–1357.

\bibitem{DKM} H.A. Dye, L.H. Kauffman, V.O. Manturov, On two categorifications of the arrow polynomial for virtual knots, {\em The Mathematics of Knots, Contributions in Mathematical and Computational Sciences} v.1 (2011), pp. 95--124.

 \bibitem {HS} J.~Hass, P.~Scott, Shortening curves on surfaces, {\em Topology}, {\bf 33}:1 (1994), pp.\ 25--43.

 \bibitem{Kauffman}
L.\,H.~Kauffman,  Virtual knot theory, {\em European Journal of
Combinatorics}, {\bf 20}:7 (1999), pp.\ 663--690.

\bibitem{Man1} V.O. Manturov, Knot Theory, Chapman \& Hall/CRC Press (2004).

\bibitem{Man2} V.O. Manturov, Additional gradings in the Khovanov complex for thickened surfaces, {\em Dokl. Math.}, {\bf 77} (2007), pp. 368-–370.

\bibitem{Miyazawa} Y. Miyazawa, Magnetic graphs and an invariant for virtual links,
{\em J. Knot Theory Ramifications}, {\bf 15} (2006), pp.\ 1319–-1334.





\end{thebibliography}
\end{document}